\newcommand{\ff}{\mathbb F}

\newcommand{\F}{\mathbb{F}}

\documentclass[a4paper,12pt]{elsarticle}

\usepackage{graphicx}
\usepackage[all]{xy}
\usepackage[english]{babel}
\usepackage{amssymb,amsmath,amsthm}
\usepackage[usenames,dvipsnames]{xcolor}
\usepackage{breqn,cancel}
\usepackage{enumerate}
\usepackage{pstricks,pst-plot}
\usepackage{verbatim}
\usepackage{hyperref}

\hypersetup{colorlinks=true}

\newtheorem{theorem}{Theorem}[section]

\newtheorem{lemma}[theorem]{Lemma}

\newtheorem{corollary}[theorem]{Corollary}

\newdefinition{definition}{Definition}
\newdefinition{example}{Example}
\newdefinition{remark}{Remark}[section]

\title{Frobenius nonclassical components of  curves with separated variables}

\begin{document}
\makeatletter
\def\ps@pprintTitle{%
  \let\@oddhead\@empty
  \let\@evenhead\@empty
  \let\@oddfoot\@empty
  \let\@evenfoot\@oddfoot
}
\makeatother

%
%

\begin{frontmatter}

\title{Frobenius nonclassical components of  curves with separated variables}

\author{Herivelto Borges}

\address{Universidade de S\~ao Paulo, Inst. de Ci\^encias Matem\'aticas e de Computa\c c\~ao, S\~ao Carlos, SP 13560-970, Brazil.}

\date{Nov  2013}

\begin{abstract} 
We establish a relation  between minimal value set polynomials defined over $\F_q$ and   certain 
$q$-Frobenius nonclassical curves.  The connection leads to a characterization of the  curves of type $g(y)=f(x)$, whose irreducible components are $q$-Frobenius nonclassical. An immediate consequence will be  the realization of  rich sources of new $q$-Frobenius nonclassical curves.
 \end{abstract}


\end{frontmatter}

\section{Introduction}\label{intro}

Let $p$ be a prime number, and $\F_q$ be the field with $q=p^s$ elements. An irreducible plane curve $\mathcal{F}: F(x,y)=0$ defined over $\F_q$ is called $q$-Frobenius nonclassical if 
\begin{equation}\label{frob}
F(x,y) \hspace{0.3 cm}\text{divides}\hspace{0.3 cm} (x^q-x)\frac{\partial F}{\partial x}+(y^q-y)\frac{\partial F}{\partial y}.
\end{equation}

Otherwise,  $\mathcal{F}$ is called $q$-Frobenius classical. Note that  the previous condition above has a geometric meaning:  the $\F_q$-Frobenius map takes any simple point $P$ of  $\mathcal{F}$ to the tangent line of $\mathcal{F}$ at $P$. 

Frobenius nonclassical curves were first introduced in  the work of St\"ohr and Voloch  \cite{SV}. It is well known that such curves potentially have many rational points and interesting arithmetic and geometric properties (\cite{HV},\cite{SV}). This fact, along with other related results (\cite{Multi-Frob},\cite{garcia-ynf}, \citep*{arcos-italianos},\cite{Hefez1}), makes the characterization of these curves highly desirable. This  common sense is  reaffirmed by the following quotation from the recent book by Hirschfeld, Korchm\'aros and Torres \cite[p. 407]{Livro-HKT}:  `` it is hard to find Frobenius nonclassical curves. What emerges is that they are rare but important curves". 

With regard to the number of rational points, one important result is the following (cf.  \cite[Theorem 2.3]{SV}).
\begin{theorem}[St\"ohr-Voloch]\label{H-V} Let $\mathcal{X}$ be an irreducible plane curve of degree $d$ and genus $g$ defined over $\F_q$. If  $N:=\#\mathcal{X}(\F_q)$ is the number of $\F_{q}$-rational points on $\mathcal{X}$, then
\begin{equation}\label{SV-bound}
N\le \dfrac{\nu (2g-2) +(q+2)d}{2},
\end{equation}
where \begin{equation}\label{eq:cases}
\nu =\begin{cases}
p^h, \hspace{0.3cm}\text{for some $h\geq 1$}, &  \text{ if  $\mathcal{X }$ is $q$-Frobenius nonclassical}\\
1, & \text{ otherwise.}\\
\end{cases}
\end{equation}
\end{theorem}

Thus if we are able to identify  the $q$-Frobenius nonclassical curves, we will be left with the remaining curves for which a better upper bound holds (inequality (\ref{SV-bound}) with $\nu=1$). At the same time, the set of $q$-Frobenius nonclassical curves provides  a potential source of curves with many points. Therefore, in  light of Theorem \ref{H-V}, characterizing $q$-Frobenius nonclassical curves may offer a two-fold benefit.

Examples of $q$-Frobenius nonclassical curves are the Fermat curves 
\begin{equation}\label{fermat}
x^{\frac{q-1}{q'-1}}+y^{\frac{q-1}{q'-1}}+1=0,
\end{equation}
where $\F_{q'} \subseteq \F_{q} $ (see \cite[Theorem 2]{G-V})  Note that, in particular,  the Hermitian curve (case $q'=\sqrt{q}$) is $q$-Frobenius nonclassical.  Additional examples  can be found in the literature (\cite{Multi-Frob},\cite{garcia-ynf},\cite{HV}).
 
The purpose  of this paper is to present a new  connection between certain $q$-Frobenius nonclassical curves and minimal value set polynomials, i.e., polynomials
$f(x)\in \F_q[x]$ for which $V_f=\{f(\alpha): \alpha \in \F_q\}$ has size $\left\lceil q/\deg f\right\rceil$.  This relation leads us to a characterization of the curves $f(x)=g(y)$ for which all  irreducible components are $q$-Frobenius nonclassical. For a prototype of this connection, note that 
$$f(x)=x^{\frac{q-1}{q'-1}} \text{ and } g(y)=-(y^{\frac{q-1}{q'-1}}+1)$$
 are minimal value set polynomials over $\F_q$, with $V_f=V_g=\F_{q'}$, and that $f(x)=g(y)$ is the  $q$-Frobenius nonclassical curve given in  (\ref{fermat}).  
 
 A consequence of this correspondence will be the realization of new sources  of  Frobenius nonclassical curves. It is worth noting that,  as one would expect, some curves in this new family will have many rational points. For instance,  we will see that the so-called generalized Hermitian curve,  the  curve over $\F_{q^k}$ ($k\geq 2$) given by

\begin{equation}\label{curvaGS}
\mathcal{GS}: y^{q^{k-1}}+\cdots+y^q+y=x^{q+1}+x^{1+q^2}+\cdots+x^{q^{k-1}+q^{k-2}},
\end{equation}
is  Frobenius nonclassical. The curve (\ref{curvaGS}), introduced by Garcia  and Stichtenoth in \cite{Garcia_A_class_of_poly_1999424}, has genus  $g=(q^{k-1}-1)q^{k-1}/2,$ and $N=q^{2k-1}+1$  $\F_{q^k}$-rational points.
Additional arithmetic properties of the curve $\mathcal{GS}$ (\cite{bulygin_generalize_hermitian},\cite{munuera_sepulv_Generl_hermitian_code}, \cite{Munuera_sepulv_Algebraic_codes_castle}) make it suitable for  construction of algebraic geometric codes with good parameters.

It should be mentioned that, using the characterization established in this paper, an alternative generalization of the Hermitian curve (with an even better ratio $N/g$) can be constructed. This will be  the content of a subsequent paper.

The present paper is organized as follows. In Section \ref{secao2}, we recall the main facts and results related to minimal value set polynomials. In Section \ref{secao3}, we establish a connection between  Frobenius nonclassical curves and minimal value set polynomials, and prove Theorem \ref{main0} and Corollary \ref{main}, the main results of this paper. In Section \ref{secao4}, we consider the minimal value set polynomials $F\in \F_q[x]$ for which $|V_F|\leq 2$. As observed  by Carlitz et al. \cite{carlitz_pol_minimal_61},
these polynomials do not follow the general pattern. Nonetheless, we will be able to characterize the ones that do give rise to Frobenius nonclassical curves.  In  Section \ref{secao5}, we make use of the preceding results to characterize Frobenius nonclassical curves of type $y^n=f(x)$. This will incorporate Garcia's results in \cite{garcia-ynf}.  In  Section \ref{secao6},  we provide additional examples and briefly discuss some problems related to Frobenius classicality.
In particular,  we  answer  a question raised by Giulietti et al. \cite{arcos-italianos} regarding the arc property of Frobenius nonclassical curves.

\text{}\\

{\bf \large {Notation}} 
\begin{itemize}
\item $q$ denotes a power of a prime $p$.

\item $\overline{\F_q}$ denotes the algebraic closure of $\F_q$.

\item If $f(x,y) \in \F_q[x,y]$, and the affine curve $\mathcal{F}: f(x,y)=0$ is irreducible, we denote by  $\mathcal{F}(\F_q)$ the set of $\F_q$-rational points on the projective closure of $\mathcal{F}$.

\item We will  omit the $q$-part in the name {\it $q$-Frobenius (non)classical}, when the finite field $\F_q$ is clear or irrelevant in the  context.

\item For any polynomial $f\in \F_{q}[x]$, the symbol $f'$ will denote the formal derivative of $f$.
\end{itemize}


\section{Minimal value set polynomials}\label{secao2}
For any nonconstant polynomial $F\in\F_q[x]$, let $V_F=\{F(\alpha): \alpha \in \F_q\}$ be its value set. One can easily verify that $V_F$ satisfies

\begin{equation}\label{mvsp-def}
\left\lfloor\dfrac{q-1}{\deg F}\right\rfloor +1 \leq |V_F| \leq q.
\end{equation}

\begin{definition}\label{def:mvsp}
A polynomial $F\in\F_q[x]$ is called \emph{minimal value set polynomial} (shortened to $MVSP$ ) if 
$|V_F|$ attains the lower bound in $(\ref{mvsp-def})$.
\end{definition}

Despite significant past results in \cite{carlitz_pol_minimal_61} and \cite{mills_pol_minimal_64}, and recent progress presented  in \cite{borges_con_charac_mvsp}, the  complete  characterization of MVSPs is still an open problem. 

A fundamental result concerning  these polynomials is the following (cf. \cite[Theorem 1]{mills_pol_minimal_64}).

\begin{theorem}[Mills] \label{main1} Let $F \in \F_{q}[x]$ be  nonconstant polynomial, and   consider the following.

\begin{enumerate}[\rm(i)]
\item Let $V_F=\{\gamma_0,\gamma_1,\ldots,\gamma_r\} \subseteq \F_{q}$ be the value set of $F$.
\item For each $i\in \{0,\ldots,r\}$, set $L_i:=\gcd(F-\gamma_i,x^q-x)$.

\item Suppose  $\gamma_i$ are arranged in such a way that $\deg L_0\leq \deg L_i$, $1\leq i \leq r$. 

\end{enumerate}

If  $F$ is an MVSP and $r>1$,  then there exist positive integers $v,m,k$; a polynomial $N\in\F_q[x]$, and $ \omega_0,\omega_1,\ldots,\omega_m\in \F_q$, with   $0\neq \omega_0$ and $\omega_m=1$, such that
\begin{enumerate}[\rm(a)]
\item $v\mid (p^{k}-1)$,  $1+vr=p^{mk}$, $L_0\nmid N$, and $L_0'$  is a $p^{mk}$-th power.
\item $F=L_0^{v}N^{p^{mk}}+\gamma_0.$
\item $\prod \limits_{i=1}^{r}(x-\gamma_i+\gamma_0)=\sum \limits_{i=0}^{m}\omega_i x^{(p^{ki}-1)/v}$.
\item $\sum \limits_{i=0}^{m}\omega_i L_0^{p^{ki}}N^{p^{mk}(p^{ki}-1)/v}=-\omega_0(x^{q}-x)L_0'.$
\end{enumerate}
   
\end{theorem}

\noindent In the remainder of this section, we provide additional results on MVSPs that bear upon the sections that follow. The next  theorem  will be a key ingredient. It is a slightly extended version of   \cite[Theorem 3.1]{borges_con_charac_mvsp}, which in turn  is  partially  derived from results in \cite{mills_pol_minimal_64}.

\begin{remark}\label{troca}Note that if  $a,b \in \F_{q}$  are distinct, and 
$$S_{\{a,b\}}=\Big\{F\in \F_{q}[x]: V_{F}=\{a,b\}\Big\},$$
 then a map $ S_{\{a,b\}} \rightarrow S_{\{0,1\}}$ given by $F \mapsto \frac{1}{a-b}(F-b)$
is a bijection. We use this fact in some of our later proofs.
\end{remark}

\begin{theorem}\label{Mills}
Let $F \in \F_q[x]$ be a polynomial of degree $d\geq 1$. If there exists $\theta \in \F_q^*$, and a monic polynomial $T \in \overline{\F_q} [x]$ such that 
\begin{equation}\label{mvsp}
T(F)=\theta(x^q-x)F^{\prime},
\end{equation}
then $T= \prod\limits_{\gamma_i \in V_F}(x-\gamma_i)$ and $F$ is an MVSP. Conversely, suppose that $F$ is an MVSP and
$T= \prod\limits_{\gamma_i \in V_F}(x-\gamma_i)$. If either $|V_F|>2$ or $|V_F|=2=p$,  then there exists $\theta\in\F_q^*$ such that
(\ref{mvsp}) holds.
\end{theorem}

\begin{proof}
Set $t:=\deg T$, and let $S\subseteq \overline{\F_q} $ be the set of distinct roots of $T$. Note that 
$V_F\subseteq S$, and so $|V_F|\leq |S|\leq t$. On the other hand, equating degrees in (\ref{mvsp}), gives
$$ t\cdot d=q+\deg F^{\prime}\leq q-1+ d.$$
Thus $(t-1)d \leq q-1$, which gives  $t\leq \frac{q-1}{d}+1$, and then 
$$|V_F|\leq |S|\leq t\leq \lfloor\frac{q-1}{d}\rfloor+1 \leq |V_F|.$$
Therefore, $|V_F|=|S|= t=\lfloor\frac{q-1}{d}\rfloor+1.$ That is,  $F$ is an MVSP and $T=\prod\limits_{\gamma_i \in V_F}(x-\gamma_i)$. For the converse, if 
$|V_F|>2$,  one can readily check that  the result follows from  \cite[equation (4)]{mills_pol_minimal_64} and \cite[Lemma 1]{mills_pol_minimal_64}. For the case $|V_F|=2=p$, 
we show that the result follows from  \cite[Lemma 4.1]{borges_con_charac_mvsp}. In fact, 
from Remark \ref{troca}, we may assume $V_F=\{0,1\}$, and  then 
\citep[Lemma 4.1]{borges_con_charac_mvsp}  implies that  ${F}^2-{F}=(x^q-x){F}^{\prime}$, which completes the proof.
\end{proof}

\begin{lemma}\label{lema00}

Let $F\in \F_q[x]$ be a nonconstant polynomial, and  let 
$$V_F=\{\gamma_0,\gamma_1,\ldots,\gamma_r\}$$ be its value set. For each $\gamma_i \in V_F$, define    $F_i:=F-\gamma_i$. If  $F$ satisfies  equation (\ref{mvsp}) in Theorem $\ref{Mills}$,  then the following hold.

\begin{enumerate}[\rm(i)]
\item If $\alpha \in \overline{\F_{q}}\backslash \F_{q}$ is  a root of $F_i$ of multiplicity $k\geq 1$, then
$p\mid k$. 
\item If $k\geq 1$ is the multiplicity of an $\F_q$-root of any $F_i$, then $ T'(\gamma_i)=-\theta k$. In particular,  $p\nmid k$.
\item $F'\neq 0$, and if $r>0$, then there exists $\gamma_i\in V_F$  such that  $T'(\gamma_i)=-\theta$.  
\item  If $r>0$, then  F''=0 \text{ if and only if  } T'  \text{ is constant. }
\end{enumerate}
\end{lemma}

\begin{proof}  
Without loss of generality, we prove assertions (i) and (ii) for  $F_0=F-\gamma_0$. Suppose  $F_0=\prod\limits_{j=1}^d(x-a_j)^{k_j}$,  where  $a_j\in \overline{\F_{q}}$ are distinct, and $k_j\geq 1$  are integers. To prove (i), we may  assume that $a_1 \notin \F_q$, and then from
\begin{equation} \label{poly}
T(F)=F_0F_1\cdots F_r=\theta (x^q-x)F'(x)
\end{equation}
we have that   $(x-a_1)^{k_1}$ divides $F'(x)=F_0'(x)$. But since $k_1$ is the multiplicity of $a_1$, we have that  $p|k_1$.
To prove  (ii), first note that (\ref{poly}) implies that
\begin{equation} \label{deriva}
\frac{T(F)}{F_0}=\theta (x^q-x)\frac{F_0'(x)}{F_0(x)}=\theta (x^q-x) \sum\limits_{j=1}^d\frac{k_j}{x-a_j}.
\end{equation}
Now if $a_{\lambda} \in \F_q$ is any root of $F_0$, then  evaluating the left and  right sides of (\ref{deriva})  at $x=a_{\lambda}$ we get  $T'(\gamma_0)=-\theta k_{\lambda}$, which provides the result. Also observe that since $T(\gamma_0)=0$ and  $T$ is separable, we have $0\neq  T'(\gamma_i)=-\theta k$ and so $p\nmid k$.
 For the third  assertion, first note that   $F'\neq 0$ is clearly given by equation (\ref{poly}). 
 For the following statement, just differentiate both sides of equation (\ref{mvsp}) in Theorem \ref{Mills}, and then evaluate at any $x=\alpha \in \F_q$ for which $F'(\alpha) \neq 0$. The existence of such $\alpha$ comes from the fact that $F'\neq 0$ and $\deg F' <q$.
 
To prove the lemma's last claim, observe that 
if $T'$ is constant then (iii)  gives  $T'=-\theta$. Thus  (iv) will follow immediately after we differentiate both sides of   (\ref{mvsp}) in Theorem \ref{Mills}.  This finishes the proof.
\end{proof}

\begin{lemma}\label{lema0}

Notation and hypotheses as in Lemma $\ref{lema00}$. Assume $r>1$ and   let  $l_i$ be the degree of $L_i:=\gcd(F_i,x^q-x)$. If  $\gamma_i \in V_F$ are  labelled in such a way that $l_0\leq l_i$ for $i=1,\ldots,r$, then

\begin{enumerate}[\rm(i)]
\item The multiplicities of all $\F_q$-roots of $F_1,\ldots,F_r$ reduce to $1 \mod p$.
\item $\theta=-T(\gamma_i)$ for  all $\gamma_i \in V_F\backslash \{\gamma_0\}$.
\end{enumerate}
\end{lemma}

\begin{proof}

From Theorem \ref{Mills}, we have that $F$ is an MVSP. Clearly $F_i=L_iU_i$, for some polynomial $U_i$, $i=0,\ldots,r$. Now the first assertion  is given directly by   \cite[Lemma 2]{mills_pol_minimal_64} (see also notation between (2)  an (3) therein), followed by  \cite[condition (11)]{mills_pol_minimal_64} and  \cite[Lemma 3]{mills_pol_minimal_64}. Item  (ii) is given directly from  assertions (ii) and (i)  of our  Lemmas \ref{lema00} and \ref{lema0}, respectively. This  gives the result.

\end{proof}


\subsection{Minimal value sets polynomials $F\in \F_{q^k}[x]$ with $V_F=\F_q$}

The Theorem 4.7 in \cite{borges_con_charac_mvsp} gives a complete characterization of MVSPs $F\in \F_{q^k}[x]$ for which  $V_F=\F_q$. To enhance clarity, we state this result in a slightly different way  and provide its proof adjusted accordingly.
\begin{theorem}\label{MVSP-set}
Let $F\in \F_{q^k}[x]$ be a nonconstant polynomial.  Then $F$ is an MVSP with $V_F=\F_q$ if and only if
there exists a nonconstant $H\in \F_{q^k}[x]$ such that
\begin{enumerate}[\rm(i)]
\item the monomials of $H$ are of the form $cx^{\alpha_0+\alpha_1q^2+\cdots+\alpha_{k-1}q^{k-1}} \in \F_{q^k}[x]$,
where each $\alpha_i$ is either $0$ or $1$.

\item \begin{equation}
F=T_{k}(H) \mod (x^{q^k}-x),
\end{equation} 
where $T_{k}(x):=\sum\limits_{i=0}^{k-1}x^{q^i}$ is the trace polynomial.
\end{enumerate}
\end{theorem}

\begin{proof}
Assuming (i) and (ii), we clearly have $V_{F}\subseteq \F_q$ and 
\begin{equation}\label{grau}
\deg F\leq q^{k-1}+\cdots+q+1=\frac{q^k-1}{q-1}.
\end{equation}
Thus   \cite[Lemma 4.1]{borges_con_charac_mvsp} implies that $F\in \F_{q^k}[x]$ is an MVSP with $V_{F}=\F_q$.  Conversely, if  $F$ is an MVSP with $V_F=\F_q$, then  \cite[Theorem 4.7]{borges_con_charac_mvsp} asserts  that $F$ is a sum of polynomials of the form 
 \begin{equation}\label{eq:char_mvsp_fq}
\mathfrak{F}:=\sum_{i=0}^{t-1} \Big({m(x)}^{q^i} \mod (x^{q^n}-x)\Big),
\end{equation}
where $m(x)\in\ff_{q^{t}}[x]$ is a monomial of degree $\alpha_{n-1}q^{n-1}+\cdots+\alpha_{1}q+\alpha_0$,  $\alpha_i\in\{0,1\}$,
and $t$ is the size of the orbit of $m(x)$ under the action of $G:=Gal(\F_{q^k}|\F_q)$ on the set of monomials of $F$ (cf. \cite[Proposition 4.2]{borges_con_charac_mvsp}). Therefore, it suffices to prove (i) and (ii) for the polynomial  $\mathfrak{F}$ in (\ref{eq:char_mvsp_fq}). Note that
$\mathfrak{F}$ is $G$-invariant, that is, $\mathfrak{F}^{q^i} \mod (x^{q^k}-x)=\mathfrak{F}$ for any integer $i\geq 0$. Now if we  take  $\lambda  \in \F_{q^k}$ such that $T_{k}(\lambda)=1$, and define $H:=\lambda \mathfrak{F}$,  we obtain
$$T_{k}(H)\mod (x^{q^k}-x)=T_{k}(\lambda \mathfrak{F})\mod (x^{q^k}-x)=T_{k}(\lambda )\cdot \mathfrak{F}=\mathfrak{F},$$
and the result follows.
\end{proof}

\begin{remark}\label{W-remark} Note that from Theorem \ref{MVSP-set} the set 

\begin{equation}\label{W-set}
\mathcal{W}:=\{\text{MVSPs} \hspace{0.2 cm} F \in \F_{q^k}[x] : V_F=\F_q\},
\end{equation}
can be explicitly  constructed. As a matter of fact, it follows from  \cite[Theorem 4.8]{borges_con_charac_mvsp}  that the set $\mathcal{W} \cup \F_q$  is an $\F_q$-vector space of dimension $2^{k}$ (in particular, $\#\mathcal{W}=q^{2^k}-q$).  For example, if $k=2$, one  can easily  check that  this  $\F_q$-vector space is given by
\begin{equation}\label{vspace}
\mathcal{W}\cup \F_q=\langle 1,x^{q+1},x+x^q, \lambda x+(\lambda x)^q\rangle,
\end{equation}
  for any fixed  $\lambda \in \F_{q^2}\backslash \F_q$. We will address the set $\mathcal{W}$ again in Section \ref{secao4}.
\end{remark}

We finish this section with a slight extension of  \cite[Proposition 2.5]{borges_con_charac_mvsp}.
\begin{lemma}\label{lema-linear}
Let $F\in \F_q[x]$ be an MVSP such that $\deg F\leq \sqrt{q}$. If $G\in \F_q[x]$ is an MVSP  with $\deg G\leq \sqrt{q}$ and $V_F=V_G$,  then  $G=F(ax+b)$ for some $a,b\in \F_q$.
\end{lemma}
\begin{proof} Note that when $|V_F|>2$ the result is given by \cite[Proposition 2.5]{borges_con_charac_mvsp}.  Since $q>4$ implies
$$|V_F|=1+\lfloor \frac{q-1}{\deg F}\rfloor=\lceil\frac{q}{\deg F}\rceil\geq \lceil \sqrt{q}\rceil>2,$$
 we may assume $ q\leq 4$ and $|V_F|\leq 2$.  If $q<4$, then the condition $\deg G, \deg F\leq \sqrt{q}$  implies that $F$ and $G$ are linear, and the result follows trivially. Therefore, the only case we are left with is
$$q=4 \text{  and  }  \deg G=\deg F=|V_F|=|V_G|=2.$$
By Remark \ref{troca},  we may assume $V_F=V_G=\F_2$. Thus $(\ref{vspace})$ implies
that $F(x)=(\lambda x)^2+\lambda x+\alpha$ and $G(x)=(\gamma x)^2+\gamma x+\beta$,
for some $\lambda,\gamma  \in \F_{4}^*$ and $\alpha,\beta \in \F_2$. Now  taking
$b\in \F_{4}$  such that $F(b)=\beta$, and $a=\gamma/\lambda$, we have 
$G=F(ax+b)$, as claimed.
\end{proof}


\section{The main results}\label{secao3}

The primary goal   of this section is to prove Theorem \ref{main0} and Corollary \ref{main}, which establish a relation between certain Frobenius nonclassical curves and  MVSPs.  To  this end, we start by presenting some preliminary results.

The following result corresponds to \cite[Theorem 3.2]{Fried}.

\begin{lemma}[Fried-MacRae]\label{lema1.9}
Let $K$ be an arbitrary field. Let  $f(x)$,$g(x)$,$a(x)$ and $b(x)$ be nonconstant  polynomials in $K[x]$.  The polynomial $f(x)-g(y)$ is a factor of $a(x)-b(y)$ if and only there exists a  polynomial  $T\in K[x]$ such that
$$T(f(x))=a(x) \text{  and  } T(g(y))=b(y).$$
\end{lemma}

\begin{lemma}\label{extra}
Let $f(x)$ and $g(y)$ be nonconstant  polynomials defined over $\overline{\F_q}$ such that $F(x,y)=f(x)-g(y)\notin \overline{\F_q}[x^p,y^p]$. If $F=\prod\limits_{i=1}^r F_i$ is the factorization of $F$ into irreducible factors, then the $F_i$ are pairwise coprime. 
\end{lemma}

\begin{proof}
Write $F(x,y)=f(x)-g(y)=\prod\limits_{i=1}^r F_i^{m_i}$, where  $(F_i,F_j)=1$ for $i\neq j$, and $m_i\geq 1$ are integers. Without loss of generality, suppose that $m_1>1$, $f^{\prime}(x)\neq 0$, and then write 
\begin{equation} \label{facto}
f(x)-g(y)=F_1^{m_1}G
\end{equation}
for some $G\in \overline{\F_q}[x,y]$. Differentiating both sides of (\ref{facto}) with respect to $x$ gives $f^{\prime}(x)=F_1^{m_1-1}(m_1\frac{\partial F_1}{\partial x}G+F_1\frac{\partial G}{\partial x})$. Therefore, 
$F_1$ divides $ f^{\prime}(x)\neq 0$, that is, $F_1$ is a nonconstant polynomial in $x$ only. Thus, since $F_1$ divides $f(x)-g(y)$, it follows that $g(y)$ is constant, a contradiction. This completes the proof.
\end{proof}

\begin{lemma}\label{components}
Let  $F(x,y) \in\overline{\F_q}[x,y]$ be a nonconstant polynomial, and write   $F=\prod\limits_{i=1}^r F_i$ where each  $F_i$ is irreducible. If the $F_i$  are pairwise coprime,  then the following are equivalent.
\end{lemma}
\begin{enumerate}[\rm(i)]
\item $F$ divides $(x^q-x)\frac{\partial F}{\partial x}+(y^q-y)\frac{\partial F}{\partial y}$.
\item $F_i$ divides $(x^q-x)\frac{\partial F_i}{\partial x}+(y^q-y)\frac{\partial F_i}{\partial y}$ for all $i=1,\ldots,r$.
\end{enumerate}

\begin{proof}
For each $i \in \{1,\ldots,r\}$, define $H_i:=\prod \limits _{j\neq i}F_j$ and write  $F=F_iH_i$.  Computing $\frac{\partial (F_iH_i)}{\partial x} $ and $\frac{\partial (F_iH_i)}{\partial y}$ and multiplying the results by $x^q-x$ and 
$y^q-y$, respectively, yields the identity $$(x^q-x)\frac{\partial F}{\partial x}+(y^q-y)\frac{\partial F}{\partial y}=$$
\begin{equation} \label{frobunion1}
\Big((x^q-x)\frac{\partial F_i}{\partial x}+(y^q-y)\frac{\partial F_i}{\partial y} \Big)H_i +\Big((x^q-x)\frac{\partial H_i}{\partial x}+(y^q-y)\frac{\partial H_i}{\partial y}\Big)F_i.
\end{equation}
The equality (\ref{frobunion1}) clearly gives that (ii) implies (i). The converse also follows from   (\ref{frobunion1}), when we use the fact that $(F_i,F_j)=1$, for $i\neq j$, implies $(F_i,H_i)=1$.

\end{proof}

\begin{theorem}\label{main0}
Let $f(x),g(x) \in \F_q[x]$ be nonconstant polynomials  such that  $f(x)-g(y)\notin \F_q[x^p,y^p]$. Suppose that the irreducible components of $\mathcal{F}: f(x)=g(y)$ are defined over $\F_q$.  The   irreducible  components of $\mathcal{F}$ are $q$-Frobenius nonclassical if and only if there exist a monic polynomial $T\in \F_q[x]$ and a constant $\theta \in \F_q^*$, such that
\begin{equation}\label{FNC&T}
T(f(x))=\theta(x^q-x)f'(x) \hspace{0.1 cm}\text{ and }\hspace{0.1 cm} T(g(y))=\theta(y^q-y)g'(y).
\end{equation}
\end{theorem}
\begin{proof}

\noindent We begin by proving (\ref{FNC&T}). Set  $f(x)-g(y)=\prod\limits_{i=1}^r F_i$, where $F_i\in \F_{q}[x,y]$ are absolutely irreducible, and let us assume that each  curve $F_i=0$ is $q$-Frobenius nonclassical. That is, 
$$F_i  \text{ divides } (x^q-x)\frac{\partial F_i}{\partial x}+(y^q-y)\frac{\partial F_i}{\partial y}, \hspace{0.1cm}  \text { for all } \hspace{0.1cm}  i=1,\ldots,r.$$ From Lemma \ref{extra}  the $F_i$ are pairwise coprime, and so  Lemma \ref{components} implies that 
\begin{equation}\label{fgab}
f(x)-g(y) \hspace{0.1 cm}\text{ divides }\hspace{0.1 cm}(x^q-x)f'(x)-(y^q-y)g'(y).
\end{equation}
Now with the  conditions $f(x)-g(y)\notin \F_q[x^p,y^p]$  and (\ref{fgab}),  the statement  (\ref{FNC&T}) follows directly from Lemma \ref{lema1.9}. The converse follows easily from
the fact that $f(x)-g(y)|T(f(x))-T(g(y))$ together with Lemmas \ref{extra} and \ref{components}.

\end{proof}

\begin{corollary}\label{main}
Let $f(x),g(x) \in \F_q[x]$ be nonconstant polynomials  such that  $f(x)-g(y)\notin \F_q[x^p,y^p]$. Suppose that the irreducible components of $\mathcal{F}: f(x)=g(y)$ are defined over $\F_q$.  If  the irreducible component of $\mathcal{F}$  are $q$-Frobenius nonclassical, then $f(y)$ and $g(x)$ are  MVSPs  with $V_f=V_g$. Conversely, suppose that  $f(x)$ and $g(y)$ are  are  MVSPs with $V_f=V_g $.  If $|V_f|>2$ or  $|V_f|=2=p$, then the irreducible  components of $\mathcal{F}$ are $q$-Frobenius nonclassical.
\end{corollary}

\begin{proof}
This is an immediate consequence of Theorems  \ref{Mills} and \ref{main0}.
\end{proof}

\begin{remark} The converse of Corollary \ref{main} in the cases $|V_f|=1$ and 
$|V_f|=2<p$  will be detailed addressed in Section \ref{secao4}.
\end{remark}


\subsection{Some Consequences}

Next, we point out some facts that follow immediately  from Corollary \ref{main}.

\begin{corollary} If   $H=\{a_ix-b_i| i=1,\ldots,n\}$ is a subgroup of $Aut(\F_q(x))$  and $f(x)=\prod\limits_{i=1}^n(a_ix-b_i)$, then $f(x)$ is an MVSP. 
\end{corollary}

\begin{proof}
Since $(H,\circ)$ is a group, the polynomial of degree $n$
$$F(x,y)=f(x)-f(y)=\prod\limits_{i=1}^{n} (a_ix-b_i)-\prod\limits_{i=1}^{n} (a_iy-b_i)$$
is such that $F(x,a_ix+b_i)=0$ for all $i=1,\ldots,n$. Therefore, the plane curve $F(x,y)=0$ is the union of $n$ distinct lines
$y=a_ix+b_i$. In particular, $f^{\prime}(x) \neq 0$ and  the $n$ irreducible components of $F(x,y)=0$ are Frobenius nonclassical curves. Therefore Corollary \ref{main}  gives the result.
 \end{proof}

Hefez and Voloch \cite [Proposition 6]{HV}  have  proved that if a $q$-Frobenius nonclassical curve of degree $d>1$ is nonsingular, then $d\geq \sqrt{q}+1$.
Next, we show that the nonsingularity condition can  be dropped for the curves considered here.

\begin{corollary}\label{maincor2} Any nonlinear $q$-Frobenius nonclassical curve $\mathcal{F}:f(y)=g(x)$ has degree $d\geq \sqrt{q}+1$.
\end{corollary}
 
\begin{proof}

From Corollary \ref{main},  $f$ and $g$ are MVSPs with $V_f=V_g$. Now suppose $\deg \mathcal{F} \leq \sqrt{q}$, i.e., $\deg f, \deg g \leq \sqrt{q}$. Thus Lemma \ref{lema-linear} implies  $g(x)=f(ax+b)$ for some $a,b \in \F_q$. Hence the line $y=ax+b$ is a component of the curve $\mathcal{F}$, which contradicts its irreducibility. This finishes the proof.

\end{proof}

Recall from Remark \ref{W-remark} that $\mathcal{W}$ denotes the set of MVSPs in $\F_{q^k}[x],$ whose  value set is $\F_q$. From the characterization given by Corollary \ref{main}, the set $\mathcal{W}$  turns into a productive source of new Frobenius nonclassical curves. In other words, we have the following.

\begin{corollary}\label{maincor3} Let $f,g$ be polynomials in $\mathcal{W}$. If the irreducible components of $\mathcal{F}:f(y)=g(x)$ are defined over $\F_{q^k}$, then all such components are $q^k$-Frobenius nonclassical curves.
\end{corollary}
  
Taking into account Corollary \ref{maincor3}, it follows that all irreducible curves of type $y^{\frac{q^k-1}{q-1}}=f(x)$, where $f(x) \in \mathcal{W} $, are $q^k$-Frobenius nonclassical. In particular, the so-called Norm-Trace curve
\begin{equation}\label{Norma-Traco}
y^{q^{k-1}+\cdots+q+1}=x^{q^{k-1}}+\cdots+x^q+x
\end{equation}
is $q^k$-Frobenius nonclassical. The Frobenius nonclassicality of cyclic coverings of $\mathbb{P}^1$ (e.g. curve (\ref{Norma-Traco}))
will be  the focus of Section \ref{secao5}.


\section{Frobenius nonclassicality in the cases $|V_f|\leq 2$ }\label{secao4}

Recall from Theorem \ref{Mills} that all nonconstant polynomials $f\in \F_{q}[x]$ satisfying equation (\ref{mvsp})   are MVSPs.
The converse  also holds for MVSPs for which either $|V_f|>2$ or $|V_f|=2=p$.  However, if $|V_f|=1$ or $|V_f|=2<p$,
one can easily find examples for which equation  (\ref{mvsp}) fails, i.e.,
\begin{equation}\label{failure}
T(f)=(x^q-x)h(x), \text{  but  } h(x)\neq \theta\cdot f'(x) \text{  for all   } \theta \in \F_q^{*}.
\end{equation}
              
 It is easy to check that the polynomials $f(x)=(x^q-x)x$ and $g(x)=\frac{x^q-x}{x^p-x}+x^{q-1}-1$ (for $p>2$) are examples of such a failure. As remarked by Carlitz et al. \cite{carlitz_pol_minimal_61},    the  MVSPs $f$ with  $|V_f|\leq 2$ do not fit the general pattern. That is the reason  why they were left out in  Theorems \ref{Mills}.
 
 The objective of this section is to characterize the MVSPs $f\in \F_q[x]$, where   $|V_f|\leq 2$, for which the polynomial $h(x)$ in (\ref{failure}) is indeed $\theta\cdot f'(x)$ for some $\theta \in \F_q^{*}.$ The results here will complement the converse of Corollary \ref{main}. That is, we will  complete the  characterization of our  Frobenius nonclassical curves in terms of MVSPs. 
 

\subsection{Case $|V_f|=1$}

It is straightforward to see that a nonconstant  polynomial $f\in \F_q[x]$ satisfies $|V_f|=1$ if and only if
\begin{equation}\label{so1}
 f=(x^q-x)r(x)+\alpha,
\end{equation}         
where     $r\in \F_q[x]\backslash\{0\}$   and     $\alpha \in \F_q$. Out of these MVSPs, the ones that give rise to Frobenius nonclassical curves will be characterized in the next theorem. Note that from (\ref{so1}), to study the curves $f(x)=g(y)$ where $V_f=V_g$, we may assume $\alpha=0$.

\begin{theorem}\label{Vf2} 
Let $f(x)$ and $g(y)$ be nonconstant polynomials defined over $\F_q$ such that  $f(x)-g(y)\notin \F_q[x^p,y^p]$. Suppose that  the irreducible components of the curve $\mathcal{F}: f(x)=g(y)$  are defined over $\F_q$ and that  $f(x)$ and $g(y)$ are  MVSPs with $V_f=V_g =\{0\}$.  Then the irreducible  components of $\mathcal{F}$ are $q$-Frobenius nonclassical if and only if there exist positive integers $n,m$,  where $n\equiv m \mod p$, and polynomials $a(t), b(t) \in \F_q[t]$,  not divisible by $t^q-t$, such that 
\begin{equation}\label{osdois}
f(x)=(x^q-x)^{n}a(x)^p \text{  and  }  g(y)=(y^q-y)^{m}b(y)^p.
\end{equation}

\end{theorem}

\begin{proof}
It is clear that any  nonconstant polynomial $h \in \F_q[t]$  for which  $V_h=\{0\}$ can be written as  $h(t)=(t^q-t)^nu(t)$, where $n$ is a positive integer, and  $u\in \F_q[t]$ is such that  $(t^q-t) \nmid u(t)$.  Thus writing $f$ and $g$ in this way, we have 
$$f(x)=(x^q-x)^{n}u(x) \text{  and  }  g(y)=(y^q-y)^{m}v(y).$$
Now if the irreducible  components of $\mathcal{F}$ are $q$-Frobenius nonclassical, then Theorem \ref{main0} implies 
\begin{equation}\label{casox}
(x^q-x)^{n}u(x) =\theta (x^q-x)^n(-nu(x)+(x^q-x)u'(x))
\end{equation}
and 
\begin{equation}\label{casoy}
(y^q-y)^{m}v(y) =\theta (y^q-y)^m(-mv(y)+(y^q-y)v'(y))
\end{equation}
for some $\theta \in \F_{q}^*$.
Since (\ref{casox}) is equivalent to  $(1+\theta n)u(x) =\theta(x^q-x)u'(x)$
and  furthermore $(x^q-x)\nmid u(x)$, we have $u'(x)=0$ and $1+\theta n=0$. That is, $u(x)=a(x)^p$ for some $a(x) \in \F_q[x]$ and $\theta=-1/n$. Similarly, (\ref{casoy})  yields $v(y)=b(y)^p$ for some $b(y) \in \F_q[y]$ and $\theta=-1/m$, and then  $(\ref{osdois})$ follows.
 Conversely, note that $(\ref{osdois})$ implies that $f$ and $g$ satisfy equation  
$T(h(x))=\theta (x^q-x)h'(x)$ for $T(x)=x$ and  $\theta=-1/m=-1/n$. Thus  Theorem
\ref{main0} gives the result.

\end{proof}


\subsection{Case $|V_f|=2<p$}

Note that  MVSPs $f\in \F_q[x]$  with $|V_f|=2$ correspond to polynomials of degree $\leq q-1$ with value set of size two. Using  Lagrange interpolation, it follows that the polynomials $f\in \F_q[x]$, for which $V_{f}=\{\alpha,\beta\}$, are given by
$$f(x)=\alpha \sum\limits_{a\in S}\Big(1-(x-a)^{q-1}\Big)+\beta\sum\limits_{b\in \F_{q}\backslash S}\Big(1-(x-b)^{q-1}\Big),$$ where $S\subsetneq  \F_q$ is an arbitrary nonempty set (see e.g. \cite[p. 348]{Harald}). Note if   $S\subsetneq  \F_q$  is fixed, then by Remark \ref{troca} we may assume  $V_{f}=\{0 ,1\}$, and then write 

\begin{equation}\label{lagrange}
f(x)=\sum\limits_{a\in S}\Big(1-(x-a)^{q-1}\Big).
 \end{equation}
 
  We begin by providing  an alternative description for the polynomials (\ref{lagrange}).
  
\begin{lemma}\label{lemasplit}
 Let  $S \subsetneq \F_q$ be a nonempty set.  If $g(x)=\prod\limits_{a\in S}(x-a)$ and $h(x)=\prod\limits_{b\in \F_{q}\backslash S}(x-b)$, then  $f=-g'h$ is an MVSP with $V_f=\{0,1\}$. Moreover, all MVSPs $f\in \F_q[x]$, with $V_f=\{0,1\}$,  arise in  this way.
 \end{lemma}
 
 \begin{proof}
 
   Clearly $g(x)h(x)=x^q-x$, and then $g'h+gh'=-1$. The last equality implies that $f:=-g'h$ is such that  $V_f=\{0,1\}$ and $\deg f\leq q-1$. In particular, $f$ is an MVSP. It is easy to check that different subsets $S_1$ and $S_2$ of $\F_q$ will give rise to different polynomials $f_1$ and $f_2$. That is, the number of polynomials arising in this way corresponds to the number of nonempty subsets  $S \subsetneq \F_q$, which is  $2^q-2$. Obviously,   the number of polynomials given by  ($\ref{lagrange}$) is  the same. This completes the proof.
   
 \end{proof}

We now seek  an  additional condition on  the polynomial $g$ (in Lemma \ref{lemasplit})  so that
the corresponding $f\in \F_{q}[x]$ satisfies equation (\ref{mvsp}). As we will soon see, it turns out that such a condition  is precisely $g''=0$. That is, $g$ is of the form $x\cdot a(x)^p+b(x)^p \in \F_{q}[x]$. In particular, linear polynomials and arbitrary  polynomials in characteristic two will always be suitable choices for $g$.

\begin{definition}\label{tipoAB} Consider the following sets of polynomials in $\F_q[t]$.

\begin{itemize}
\item $\mathcal{A}=\Big\{\dfrac{g'}{g}(t-t^{q}) : g \text{ is a monic proper divisor of  } t^q-t \text{ and } g''=0\}.$ 
\item$\mathcal{B}=\Big\{1-f: f \in \mathcal{A}\Big\}$.\\
\end{itemize}

We say that $f(x) \in  \F_q[x]$ is  a polynomial  of   type $A$ or  $B$ if $f(t)\in \mathcal{A}$ or $f(t)\in \mathcal{B}$, respectively. Note that since $q$ is odd we have $\mathcal{A}\cap \mathcal{B}=\emptyset$.

\end{definition}

\begin{remark} It is easy to construct polynomials of type $A$, and then of type $B$. One source of such polynomials is the   set  $\mathcal{W}$ in (\ref{W-set}), as  follows from Lemma \ref{lema00} (iv).  So as long as $g \in \mathcal{W}$ is separable, its roots will lie in $\F_{q^k}$ (Lemma \ref{lema00} (i)), and then, assuming  $g$ monic,  we have that $\dfrac{g'}{g}(t-t^{q^k})$ is of type $A$. The polynomials $g=t^{q^{k-1}}+\cdots+t^q+t$ and $g=t^{\frac{q^k-1}{q-1}}-1$ are some examples. 

Alternatively, one can follow a more general procedure: Choose coprime polynomials $a(t),b(t) \in \F_{q}[t]$ such that $g:=ta^p+b^p$ is monic. Thus $g$ is separable, and  for any extension $\F_{q^s}$ containing the splitting field of $g$, the polynomial $\dfrac{g'}{g}(t-t^{q^s})$ will be of type $A$.
\end{remark}

\begin{lemma}\label{tipos} Let $f\in \F_{q}$ be an MVSP, where $V_{f}=\{0,1\}$ and $q$ is odd. Then 
$$f(f-1)=\theta (x^q-x)f', \text{ for some } \theta \in \F_{q}^*,    \text{ if and only if }  f \in \mathcal{A} \cup \mathcal{B}.$$
 Furthermore, $\theta=1$ if $f\in \mathcal{A}$, and $\theta=-1$  if $f\in \mathcal{B}$.
\end{lemma}

\begin{proof}

Suppose $f(f-1)=\theta (x^q-x)f'$ for some $\theta \in \F_{q}^*.$ Since $T'(x)=2x-1$ and $V_f=\{0,1\}$, it follows from Lemma \ref{lema00} (iii) that $\theta=\pm 1$.
\begin{enumerate}[\rm(i)]

\item Case $\theta=1$.
Suppose 
\begin{equation}\label{eq5}
f(f-1)=(x^q-x)f'.
\end{equation}
We shall prove that $f\in \mathcal{A}$. In fact, from Lemma \ref{lemasplit}, $f=(x-x^q)\frac{g'}{g}$ for some monic divisor $g$ of $x-x^q$.  Thus  $f'=\frac{(x-x^q)(g''g-{g'}^2)+g'g}{g^2}$, and  (\ref{eq5})  implies
$$(x-x^q)\frac{g'}{g}((x-x^q)\frac{g'}{g}-1)=(x^q-x)\frac{(x-x^q)(g''g-{g'}^2)+g'g}{g^2}.$$
A straightforward simplification  leads to $(x-x^q)g''g=0$, and then $g''=0$. Therefore, $f\in \mathcal{A}$.

\item Case $\theta=-1$. Note that if $f(f-1)=- (x^q-x)f'$, then $1-f$ satisfies $(1-f)\Big((1-f)-1\Big)= (x^q-x)(1-f)'$. 
Thus the case $\theta=1$ implies $1-f\in \mathcal{A}$, and then $f\in \mathcal{B}$.

\end{enumerate}

Now we  prove the converse. Suppose $f\in \mathcal{A}$, i.e., 
$$f=(x-x^q)\frac{g'}{g}, \text{ where $g$ is a monic proper divisor of } x-x^q \text{ and }  g''=0.$$ 
Therefore,  using the fact that $f(f-1)=(x^q-x)G(x)$, for some $G(x) \in \F_{q}[x]$, we obtain
\begin{eqnarray*} 
G(x) & = &\frac{1}{(x^q-x)}f(f-1)=-\frac{h'}{h}\Big((x-x^q)\frac{h'}{h}-1\Big)\\
& = &\frac{h'-(x^q-x){h'}^2}{h^2}=f'.\\
\end{eqnarray*}
A similar computation for  $f\in \mathcal{B}$ implies $G(x)=-f'$.
\end{proof}

The following result will complement Corollary \ref{main} for the case $|V_f|=2<p$.
\begin{theorem}\label{Vf2} 
Let $f(x)$ and $g(y)$ be nonconstant polynomials defined over $\F_q$ such that  $f(x)-g(y)\notin \F_q[x^p,y^p]$. Suppose that  the irreducible components of the curve $\mathcal{F}: f(x)=g(y)$  are defined over $\F_q$, and that  $f(x)$ and $g(y)$ are  MVSPs with $V_f=V_g =\{0,1\}$.  Then the irreducible  components of $\mathcal{F}$ are $q$-Frobenius nonclassical if and only if $f$ and $g$ are both of  type $A$ or of type $B$.
\end{theorem}

\begin{proof}
If the  irreducible  components of $\mathcal{F}$ are $q$-Frobenius nonclassical,  then by Theorem \ref{main0},  there exists a monic $T\in \F_{q}[x]$ and $\theta \in \F_{q}^*$ such that 
\begin{equation}\label{dois}
T(f)=\theta(x^q-x)f'(x) \text{  and  } T(g)=\theta (y^q-y)g'(y).
\end{equation}
 From Theorem \ref{Mills},  $T(x)=x(x-1)$, and then Lemma \ref{lema00} (iii) gives $\theta=\pm 1$. Therefore, Lemma \ref{tipos} implies that $f$ and $g$ are of the same type.
Conversely, if $f$ and $g$ are of the same type, then Lemma \ref{tipos} implies that (\ref{dois}) holds for
$T(x)=x(x-1)$ and  some $\theta \in \{-1,1\}$. Therefore $F=f-g$  divides 
$$(x^q-x)f'(x) -  (x^q-x)g'(y)=(x^q-x)\frac{\partial F}{\partial x}+ (x^q-x)\frac{\partial F}{\partial y},$$
and the result follows from Lemmas \ref{extra} and \ref{components}.
\end{proof}


 \section{  The curves $y^n=f(x)$}\label{secao5}
The objective of this section is to apply the previous results to  the curves $y^n=f(x)$. 
The work presented here will subsume the related results of Garcia's investigation of the 
Frobenius nonclassicality of  a class of curves of type $y^n=f(x)$ \cite{garcia-ynf}.
 The case $char (\F_q)=2$ (not covered in \cite{garcia-ynf})   will also be included.
We begin with some preliminary facts.

\begin{remark} \label{zero-mult}
Hereafter, we say that $x_0\in \F_q$ is a root of $f(x)$ of multiplicity $k=0$, if $f(x_0)\neq 0$. Note that if $\deg f<q$, then there always exists such $x_0\in \F_q$.
\end{remark}

\begin{lemma}\label{homogeniza} Consider the curve $\mathcal{F}: y^n=f(x)$, where $f(x) \in \F_q[x]$ is a polynomial of positive degree $d\leq n$ and has an root $x_0 \in \F_q$ of multiplicity $k\geq 0$. Then there exists a polynomial $g(x) \in \F_q[x]$, of degree $n-k$ such that the projective completions of $\mathcal{G}: y^n=g(x)$ and $\mathcal{F}$ are $\F_q$-projectively equivalent. In particular, if the components of $\mathcal{F}$ are $q$-Frobenius nonclassical, then so are the components of $\mathcal{G}$. 
\end{lemma}

\begin{proof}

Without loss of generality, we can assume $x_0=0$ and write  $f(x)=a_dx^d+a_{d-1}x^{d-1}+\cdots+a_kx^k$, where $k\geq 0$ and $a_k\neq0$. Homogenizing $y^n-f(x)$ w.r.t. the variable $z$ yields
$$y^n-(a_dx^dz^{n-d}+a_{d-1}x^{d-1}z^{n-d+1}+\cdots+a_kx^kz^{n-k}).$$
Interchanging $x$ and $z$, and dehomogenizing w.r.t. the variable $z$ leads to a curve $\mathcal{G}: y^n=g(x)$, where $\deg g=n-k$, as desired. Note that preceding operations correspond to an $\F_q$-projective change of coordinates. And so  the $q$-Frobenius nonclassicality of the components  is preserved.

\end{proof}

\begin{lemma}\label{galois} Let $n$ be a divisor of $q-1$ and $f(x) \in \F_{q}[x]$ be a nonconstant polynomial. If  $y^n=f(x)$ has a solution
$(x_0,y_0) \in \F_q\times \F_q^*$, then the $\F_q$-irreducible factors of $y^n-f(x)$ are absolutely irreducible.
\end{lemma}

\begin{proof} 

Let us write  $y^n-f(x)=\prod\limits_{i=1}^r F_i(x,y)$, where $F_i(x,y) \in \F_q[x,y]$ are irreducible. Considered as polynomials in the variable $y$, we clearly have that 

\begin{equation}\label{soma}
\deg _{y}F_i\geq 1 \text{ and } \sum\limits_{i=1}^r \deg _{y}F_i=n.
\end{equation}

 Since $n|(q-1)$ and  $(x_0,y_0) \in \F_q\times \F_q^*$ is such that $y_0^n=f(x_0)$,  the equation $y^n=f(x_0)$ has $n$ distinct roots in $\F_q$. Thus (\ref{soma}) implies that
each  $F_i(x_0,y)$ is a nonconstant separable polynomial whose roots lie in $\F_q$. Suppose one of the factors of $y^n-f(x)$, say $F_1(x,y)$, is not absolutely irreducible. Without loss of generality, we may also assume $F_1(P)=0$ for $P=(x_0,y_0)$. Let $G_1=\sum\alpha_{i,j}x^iy^j \in \overline{\F_q}[x,y]\backslash \F_q[x,y]$ be a factor of $F_1$ such that  $G_1(P)=0$. Clearly,
$G_2:=\sum{\alpha_{i,j}}^qx^iy^j$ is another Galois conjugate dividing $F_1$, and the fact that  $P$ is an $\F_q$-point of $G_1=0$ implies $G_2(P)=0=G_1(P)$.
Therefore, $P=(x_0,y_0)$, where  $y_0\neq 0$,  is a singular point of $y^n-f(x)=0$, which is a  contradiction to $\frac{\partial ( y^n-f(x))}{\partial y}(P)\neq 0$.

\end{proof}

\begin{theorem}\label{thm2} Let $f(x) \in \F_{q}[x]$ be a nonconstant polynomial and $n \geq 1$ be an integer such that
$y^n-f(x) \notin \F_{q}[x^p,y^p]$. Then the irreducible components of  $\mathcal{F}: y^n=f(x)$  are $q$-Frobenius  nonclassical if and only if  $n\mid q-1$ and 

\begin{equation}\label{garcia}
n\cdot f(x)(f(x)^{\frac{q-1}{n}}-1)=(x^q-x)f^\prime(x).
\end{equation}

\end{theorem}

\begin{proof} 

Note that if the  irreducible components of  $\mathcal{F}$ are $q$-Frobenius  nonclassical, then
(\ref{FNC&T}) in Theorem \ref{main0}  implies
$$T(y^n)=n\theta (y^{q-1+n}-y^n),$$
for some monic polynomial $T \in \F_q[x]$, and $\theta\in \F_q^*$. Therefore,
\begin{center}
$n|(q-1)$,  $T(x)=x(x^{\frac{q-1}{n}}-1)$  and  $\theta =1/n,$
\end{center}
and then  (\ref{garcia}) follows when (\ref{FNC&T}) is applied to $f(x)$.
Conversely, if $n|(q-1)$ and  (\ref{garcia}) holds, then one can readily verify that $y^n=f(x)$
has a solution in $\F_q\times\F_q^*$. Thus from Lemma \ref{galois}, the irreducible components 
of $\mathcal{F}$ are defined over $\F_q$. Now, since (\ref{garcia}) implies (\ref{FNC&T}) for $T(x)=x(x^{\frac{q-1}{n}}-1)$ \text{ and } $\theta =1/n$, the result follows from Theorem \ref{main0}.

\end{proof}

\begin{corollary}\label{corogeral} Notation and hypotheses as in Theorem \ref{thm2}. If the components of  $\mathcal{F}: y^n=f(x) $ are $q$-Frobenius  nonclassical, then

\begin{enumerate}[\rm(i)]
\item $p\nmid n$ and $f'\neq 0$.
\item $\dfrac{nq}{n+q-1}\leq\deg f\leq n$, and the upper bound  (resp. lower bound)   for $\deg f$ is attained if and only if $p\nmid \deg f$ (resp. $f^{\prime}$ is constant).
\item  if $f=c\prod\limits_{i=1}^{s}(x-a_i)^{k_i} \in  \overline{\F_q}[x]$, then $a_i \in \F_q$ if and only if $p\nmid k_i$.  In particular, all simple roots of $f$ lie in $\F_q$.
\item $f$ has an  $\F_q$-root,  and   if $f$ has a simple root, then $n\equiv 1\mod p$.
\item $n\equiv 1\mod p$ if and only if $f^{\prime\prime}= 0$.
\item if an $\F_q$-root of $f$ has  multiplicity $k$, where $0< k< \deg f$, then $k\leq \frac{n-1}{|V_f|}$ and $k\equiv n \mod p$.
\end{enumerate}

\end{corollary}

\begin{proof} 

\begin{enumerate}[\rm(i)]
\item From equation (\ref{garcia}),  $p\nmid n$ if and only if $f'\neq 0$. Since $n|(q-1)$ the result follows.
\item Equating degrees on both side of (\ref{garcia}) yields
 $$ (\frac{q-1}{n} +1)\deg f= \deg f^\prime+q.$$
 Using both inequalities of  $0\leq \deg  f^{\prime} \leq \deg f -1$ gives the result.
\item  Since  $f\in \F_{q}[x]$ satisfies  equation (\ref{mvsp}) in Theorem \ref{Mills}, this follows directly from  Lemma \ref{lema00} (i), (ii).

\item Since $0\in V_{y^n}=V_f$, it is obvious that $f$ has an $\F_q$-root.  Differentiating both sides of (\ref{garcia}) gives
\begin{equation}\label{diff}
 (n-1)f^{\prime}(x)(f(x)^{\frac{q-1}{n}}-1)=(x^q-x)f^{\prime\prime}(x).
\end{equation}
If $x_0 \in \overline{\F_q}$ is a simple root of $f$, then Lemma \ref{lema00} implies $x_0 \in \F_{q}$. Thus evaluating both sides of  (\ref{diff}) at a simple root of $f$, we arrive at $n\equiv 1 \mod p$.
\item Using that  $f^{\prime}\neq 0$,  the claim follows directly from equation  (\ref{diff}).
\item From Lemma \ref{homogeniza}, we may assume that $\deg f=n-k$. Now  item (ii) above implies $n-k\geq \dfrac{nq}{n+q-1}$, and then $k \leq \dfrac{n-1}{1+\frac{q-1}{n}}=\frac{n-1}{|V_f|},$ which proves the first assertion.  Also, since $\deg f=n-k<n$,  again from item (ii),  we have  $p|(n-k)$, which finishes the proof.
\end{enumerate}

\end{proof}

The following result retrieves  \cite[Theorem 2]{G-V} and the ``Remark" in \cite[p. 38]{garcia-ynf}. We will see that the hypotheses $p\neq 2$ and $p\nmid nd$, included therein, are not necessary.

\begin{corollary} Let $a,b \in \F_{q}^*$, and let   $n$ and $d$ be positive integers.  If $$\mathcal{F}: y^n=ax^d+b$$ is $q$-Frobenius nonclassical curve, then $\mathcal{F}$ is a Fermat curve of degree $n=d=\dfrac{q-1}{q^{\prime}-1}$, and $a,b \in \F_{q'}^*$.
\end{corollary}
\begin{proof}  The curves  $y^n=ax^d+b$ and  $x^d=y^n/a-b/a$ are clearly the same. Thus Corollary \ref{corogeral} (ii)  gives $n=d$. Since $p\nmid d$, the polynomial $f(x)=ax^d+b$ has no repeated roots, and from  Corollary \ref{corogeral} (v),  we have  $n\equiv 1 \mod p$. Now equation (\ref{garcia}) yields
$$(ax^d+b)(ax^d+b)^{\frac{q-1}{d}}-1)=(x^{q-1}-1)ax^d,$$
and then
$$(ax^d+b)^{\frac{q-1}{d}+1}=ax^{d+q-1}+b.$$
However, such a polynomial identity holds if and only if  $\frac{q-1}{d}+1=q^{\prime}$, that is, $d=\dfrac{q-1}{q^{\prime}-1}$. Clearly this also implies $b^{q^{\prime}}=b$ and $a^{q^{\prime}}=a$, as claimed.

\text{}\\
\end{proof}

\begin{remark}
Note that any irreducible curve $\mathcal{F}: y^n=f(x)$, defined over $\F_q$ and  not meeting the conditions/results established in this section,
must satisfy the St\"ohr-Voloch bound for $\nu=1$ (cf. (\ref{SV-bound}) in Section \ref{intro}):
\begin{equation}\label{SV-baby} 
\#\mathcal{F}(\F_q)\leq g-1+d(q+2)/2.
\end{equation}

As mentioned previously, the characterization of Frobenius nonclassical curves is  motivated in part by the need to identify the curves for which  (\ref{SV-baby}) may not
hold. This gives a rather nice  bound applicable for the remaining curves. For a simple  illustration, consider the following  curve over $\F_{5^3}$:
 $$\mathcal{F}: y^{62}=x^{62}+(x+1)^{62}+1.$$ 
 It is easy to see that $f(x)=x^{62}+(x+1)^{62}+1$ has no repeated roots. Therefore,  since $62\not \equiv 1 \mod 5$,  Corollary \ref{corogeral} (iv) fails, and so  bound (\ref{SV-baby}) holds. That is,    $$\#\mathcal{F}(\F_{125})\leq (62-1)(62-2)/2-1+62(5^3+2)/2=5766.$$
Interestingly, it turns out that $5766$ is the actual value of $\#\mathcal{F}(\F_{125})$.
\end{remark}


\subsection{ Additional Remarks}
In what follows, we  provide some additional facts related to Frobenius nonclassical curves of type $y^n=f(x)$.
Some well-known examples of   this type of Frobenius  nonclassical curve  are

\begin{itemize}
\item $y^{q+1}=x^q+x$ (the Hermitian curve over $\F_{q^2}$).
\item $x^{\frac{q^k-1}{q-1}}+y^{\frac{q^k-1}{q-1}}=1$ (the Fermat curves over $\F_{q^k}$).
\item $y^{\frac{q^k-1}{q-1}}=x^{\frac{q^k-1}{q-1}}+(x^{q^{k-1}}+\cdots+x^q+x)$, over $\F_{q^k}$  (see \cite[Remark 2.9]{arcos-italianos}).
\end{itemize}
Note that in the preceding cases, we have $n=\frac{q^k-1}{q-1}$. The next result  generalizes  these examples.

\begin{theorem}\label{cyclic}
Let $f(x)\in \F_{q^k}[x]$ be a nonconstant polynomial, and let $\mathcal{W}$ be the set of MVSPs defined in (\ref{W-set}). The components of the  curve  $y^{\frac{q^k-1}{q-1}}=f(x) $ are $q^k$-Frobenius  nonclassical if and only if $f(x) \in \mathcal{W}$. Moreover,  a  curve $y^{\frac{q^k-1}{q-1}}=f(x)$, with $f(x) \in \mathcal{W}$,  is irreducible with probability at least $1-1/q$.
\end{theorem}
\begin{proof}

The first statement follows immediately from Corollary \ref{main}, since  $g(y)=y^{\frac{q^k-1}{q-1}} \in \mathcal{W}$. 
For the second, let $f(x) \in \mathcal{W}$ be a fixed MVSP, and consider  the $q$ polynomials  $f_i:=f(x) + \alpha_i \in \mathcal{W}$, where $\alpha_i\in \F_q$. It follows from  \cite[Lemma 2.4 (ii)]{borges_con_charac_mvsp} that at most one such $f_i$ has no simple root. So out of the $\#\mathcal{W}=q^{2^k}-q$ (cf. Remark \ref{W-remark}) curves $y^{\frac{q^k-1}{q-1}}=f(x)$, with $f(x) \in \mathcal{W}$, at most  $(q^{2^k}-q)/q$ curves will be reducible. That is, the probability of being reducible is no greater than $\frac{(q^{2^k}-q)/q}{q^{2^k}-q}=1/q$, which gives the result.

\end{proof}

\begin{remark} Recall that a detailed description of $\mathcal{W}$ in the case $k=2$ is given in (\ref{vspace}). From that, it can be  verified that all  irreducible curves $y^{q+1}=f(x)$ arising from Theorem \ref{cyclic} are  $\F_{q^2}$-isomorphic to the Hermitian curve.
\end{remark}

Hefez and Voloch \cite[Theorem 1]{HV}  have proved  that if $\mathcal{F}$ a plane smooth $q$-Frobenius nonclassical curve  of degree $n$, then
\begin{equation}\label{HV}
\#\mathcal{F}(\F_q)=n(q-n+2).
\end{equation}

Next we show that for $q$-Frobenius nonclassical curves of type $y^n=f(x)$, the number  given in (\ref{HV}) is, in fact,   a lower bound for $\#\mathcal{F}(\F_q)$. So as far  as the number of  rational points is concerned, the singular  curves  $y^n=f(x)$ may  be of considerable  interest.

\begin{theorem}\label{HVH}
Let $f(x) \in \F_q[x]$ be a nonconstant polynomial and $\mathcal{F}: y^n=f(x)$ be a $q$-Frobenius nonclassical curve. 
Then
 $$\#\mathcal{F}(\F_q)\geq n(q-n+2),$$
  and equality holds if and only if $\mathcal{F}$ is smooth.
\end{theorem}

\begin{proof}

In view of (\ref{HV}), we only need to  prove that if $\mathcal{F}$ is singular, then $\#\mathcal{F}(\F_q)>n(q-n+2)$. Let $r$ be the number of distinct $\F_q$-roots of $f(x)$, so by Corollary \ref{corogeral} we have 
$$1\leq r\leq \deg f \leq n.$$
From Corollary \ref{main},  $f(x)$ and $g(y)=y^n$ are MVSPs with $V_f=V_g$. In particular, $n|(q-1)$. Thus  the $q-r$ nonroots of $f(x)$ in $\F_q$ will give rise to $n(q-r)$ affine $\F_q$-points of $\mathcal{F}$. 
Now assume that $\mathcal{F}$ is singular and $\#\mathcal{F}(\F_q) \leq n(q-n+2)$. In particular, $n(q-r)\leq n(q-n+2) $ i.e. $r\geq n-2$.

First, let us consider the case $r \in \{n-1,n\}$.  If $f(x)$ is separable, then $\mathcal{F}$ is nonsingular,  contradicting our hypothesis. Thus we may assume
$r=n-1$ and $f(x)=(x-\alpha_1)^2(x-\alpha_2)\cdots (x-\alpha_{n-1})$, where $\alpha_i \in \F_{q}$ are all distinct. From Corollary \ref{corogeral} (vi), all $\F_q$-roots of 
$f(x)$ have the same reduction modulo $p$. This  implies $f(x)=(x-\alpha_1)^2$ and $n=2$, contradicting the irreducibility of $y^n=f(x)$.

Now consider the case $r=n-2$. This  implies that  none of  the  $N=d(q-d+2)$ $\F_q$-rational points of $\mathcal{F}$ is a ramification point over a root of $f(x)$. In particular, $f(x)$ cannot have a  simple root, which is necessarily an $\F_q$-root by Lemma \ref{lema00} (i). Therefore, 
$$n\geq \deg f \geq 2r=2n-4,$$ and then $n\leq 4$. After a quick inspection, one can see that these few  small values of $n$ can be  ruled out as well, and the result follows.

%
%

\end{proof}


\section{Final Remarks}\label{secao6}
In this section, we provide some additional examples, and briefly discuss some problems related to Frobenius nonclassical curves. 

In Sections \ref{secao3} and \ref{secao4} , we offered some examples of $q^k$-Frobenius nonclassical curves $f(x)=g(y)$, where 
$f$ and $g$ are polynomials given by the set $\mathcal{W}$ in (\ref{W-set}). Given the potentially relevant properties of the curves arising in this way, it could be beneficial  to  characterize them further. One nice example is  the so-called generalized Hermitian curve 

\begin{equation}\label{curvaGS1}
\mathcal{GS}: y^{q^{k-1}}+\cdots+y^q+y=x^{q+1}+x^{1+q^2}+\cdots+x^{q^{k-1}+q^{k-2}}
\end{equation}
over $\F_{q^k}$ ($k\geq 2$), which was introduced by Garcia and Stichtenoth  \cite{Garcia_A_class_of_poly_1999424}. They   proved that  $\mathcal{GS}$ has genus $g=q^{n-1}(q^{n-1}-1)/2$ and $N=q^{2k-1}+1$ $\F_{q^k}$-rational points. Some authors have used additional  arithmetic properties of this curve to  construct algebraic geometric codes with good parameters (see \cite{bulygin_generalize_hermitian}, \cite{munuera_sepulv_Generl_hermitian_code}, \cite{Munuera_sepulv_Algebraic_codes_castle}).

To see that  $\mathcal{GS}$ is indeed $q^k$-Frobenius nonclassical, note  that  the polynomial $f(x)$ on the right side of (\ref{curvaGS1}) is defined as  $f(x)=s_2(x,x^q,\ldots,x^{q^{k-1}})$, where 
$$s_{2}(x_1,x_2,\ldots,x_n)=\sum\limits_{i<j}x_ix_j$$
 is the second elementary symmetric polynomial.
Therefore $f \in \F_{q^k}[x]$ is a polynomial of degree $q^{k-2}+q^{k-1} \leq \frac{q^k-1}{q-1}$ such that $V_F\subseteq \F_q$. Thus from  \cite[Lemma 4.1]{borges_con_charac_mvsp},  $f \in \F_{q^k}[x]$ is an MVSP with $V_f=\F_q$. Clearly,  $g(y)=y^{q^{k-1}}+\cdots+y^q+y \in \F_{q^k}[x]$ is an MVSP with
$V_g=\F_q$ as well. Hence the $q^k$-Frobenius nonclassicality of $\mathcal{GS}$ follows from
Corollary \ref{main}.

A natural question is whether the curve  $\mathcal{GS}$ can be further generalized. For instance, one way of doing that is to identify a family of polynomials $\{f_i(x)\} \subseteq \mathcal{W}$ for which some of the curves 
\begin{equation}\label{curvaGS2}
y^{q^{k-1}}+\cdots+y^q+y=f_i(x)
\end{equation}
have a good ratio $N/g$. That is to say  at least as good as the  corresponding ratio for the curve $\mathcal{GS}$. Regarding this particular class of curves, we have  made some progress which we hope to  report in the near future. However, there is certainly room for additional research,  some of which can be quite challenging. For instance, consider the irreducible curves over $\F_{q^k}$ 
$$\mathcal{F}_{a,b}: y^{\frac{q^k-1}{q-1}}=x^{\frac{q^k-1}{q-1}}+a(x^{q^{k-1}}+\cdots+x^q+x)+b,$$
where $a,b\in \F_{q}$.
It is not hard to see that computing the genus and the number of $\F_{q^k}$-rational points of $\mathcal{F}_{a,b}$ boils down to determining the cardinality $N_{k-1}(u,v)$ of 
\begin{equation}\label{Moisio}
\{\alpha \in \F_{q^{k-1}}\mid T(\alpha)=u \text { and }N(\alpha)=v\},
\end{equation}
where $u,v \in \F_q$, and $T$, $N :\F_{q^{k-1}}\rightarrow \F_q$ are the trace and norm functions, respectively.
Apart from a few particular cases, determining $N_{k-1}(u,v)$ is still an open problem.  Nicolas Katz \cite{KatzBound} used deep results from algebraic geometry to set bounds for the number $N_{k-1}(u,v)$. More recently, Moisio and Wan \cite{Katz-Moisio} used results on the zeta function of certain toric Calabi-Yau hypersurfaces to improve Katz's bound. Part of the motivation 
to determine $N_{k-1}(u,v)$ is given by its known connections with many other problems (e.g. \cite{Cohen1},\cite{Moisio1},\cite{Wan1}).  Accordingly,  this new  relation with certain Frobenius nonclassical curves establishes an additional motivation.


\subsection{Some curves $y^{q^k-1}=f(x)$}

Recall from Theorem \ref{thm2} that  the $\F_{q^k}$-Frobenius nonclassical curves of type $y^{q^k-1}=f(x)$ are irreducible  ones for which $f(x)$ satisfies $$f(x)(f(x)-1)=(x-x^{q^k})f'(x).$$ We know (see proof of Theorem \ref{thm2}) that the polynomial $f(x)$ must be of type $B$, i.e., $f=1-\frac{g'}{g}(x-x^{q^k})$ where
\begin{center}
$g$ is a monic divisor of $x^{q^k}-x$ such that $g''=0.$
\end{center}

As was noted previously, such polynomials $f \in \F_{q^k}[x]$ can be easily constructed. The following result is an explicit example arising from  this construction. 

\begin{theorem}\label{penultimo}
If $k\geq 3$, then the curve 
$$\mathcal{F}: y^{q^k-1}=1-\frac{x^{q^k}-x}{x^q-x}$$ is $\F_{q^k}$-Frobenius nonclassical of genus
$$g(\mathcal{F})=\frac{(q^k-2)(q^{k-1}-1)-(q+1)(q-2)}{2},$$
and has at least $(q^k-1)(q^k-q)$ $\F_{q^k}$-rational points. Moreover, if $q=2$ then its number of rational points
is  $(2^k-1)(2^k-2)+3$.
\end{theorem}

\begin{proof}
Note that $f(x)=1-\frac{x^{q^k}-x}{x^q-x}$ is the polynomial  of degree $\deg f=q^k-q$ given by $$f(x)=(x^q-x)^{q-1}\prod\limits_{\alpha_i\in \F_{q^{k-1}}\backslash \F_q}(x-\alpha_i)^q.$$
Since $k\geq3$, there exist two roots of $f(x)$ whose multiplicities are coprime. Therefore,
the curve $\mathcal{F}$ is irreducible. The genus follows directly from the  Hurwitz-Zeuthen formula (see e.g. \cite{FunctionFields}).

Note that for $g(x)=x^q-x \in \F_{q^k}[x]$, we have  $\frac{g'}{g}(x^{q^k}-x)=\frac{x^{q^k}-x}{x^q-x} \in \mathcal{A}$, and then $f:=1-\frac{x^{q^k}-x}{x^q-x} \in \mathcal{B}$. Therefore,
$\mathcal{F}$ is $\F_{q^k}$-Frobenius nonclassical.
The first assertion about the number of $\F_{q^k}$-rational points, follows  (similarly to the proof of Theorem \ref{HVH}) directly from the fact that $f(x)$ has exactly $q$ $\F_{q^k}$-roots.
If  $q=2$, the three additional points come from  ramification points over the places $P_{x},  P_{x-1}$ and $P_{\infty}$. This finishes the proof.

\end{proof}

\begin{remark} It can be checked that some  of  the current records of curves with many points, listed at http://www.manypoints.org, are held by Frobenius nonclassical curves. For an example,  note that  for the case $q=2$ in Theorem \ref{penultimo}, the values $k=3$ and $k=4$
yield $(g(\mathcal{F}), \#\mathcal{F}(\F_{8}))=(9,45)$ and $(g(\mathcal{F}), \#\mathcal{F}(\F_{16}))=(49,213)$, repectively. Both  cases  are current records listed  at http://www.manypoints.org.

\end{remark}

We turn our attention to an object in Finite Geometry that  was investigated in connection with Frobenius nonclassical curves in \cite{arcos-italianos}.

An $(N,d)$-arc is a subset of $N$ points in $PG(2,q)$ with at most $d$ points on any line and $d$ on some line.
The $(N,d)$-arc  is called {\it complete} if it is not contained in an $(N+1,d)$-arc. When $\mathcal{F}$ is a projective plane curve of degree $d$, defined over $\F_q$, that intersects at least one line in $d$ distinct $\F_q$-points then $\mathcal{F}(\F_q)$, the set of $\F_q$-points of $\mathcal{F}$, is an example of $(N,d)$-arc. If such $(N,d)$-arc is complete, we say that $\mathcal{F}$ has the {\it arc property.}

In \cite{arcos-italianos}, the authors proved the arc property  for several $q$-Frobenius nonclassical curves and raised the question of whether or not all $q$-Frobenius nonclassical curves have the arc property. In \cite{borges1}, we gave a negative answer to this question using  a particular  singular curve. The next theorem  will provide additional counter-examples, but now arising from  nonsingular curves.

\begin{theorem} If  $q$ is a power of $2$ and $k\geq 3$ is  an integer, then the curve
$$\mathcal{F}: y^{q^k-1}=\dfrac{x^{q^k}-x}{x^q-x}$$  is  $q^k$-Frobenius nonclassical. Moreover, if  $q=2$   then $\mathcal{F}$  is  smooth and does not have the arc property. 
\end{theorem}

\begin{proof} Note that $f(x)=\dfrac{x^{q^n}-x}{x^q-x}$ has no repeated roots, and so the curve  $\mathcal{F}$ is irreducible.  Its $q^k$-Frobenius nonclassicality follows directly from (\ref{garcia}) in Theorem \ref{thm2}.
It is clear that for $q=2$,  $\mathcal{F}$ is a smooth curve of degree $d=q^k-1$.  In this case, the Hefez-Voloch formula in (\ref{HV})  gives $$\#\mathcal{F}(\F_{q^k})=d(q^k-d+2)=3d.$$ 
Now one can easily  check that these  $3d$ rational points of $\mathcal{F}$ lie on the union of lines given by $xy(x-z)=0$, with $d=q^k-1$ poinst on each line. Let $P\in PG(2,q^k)$ be a point on the complement  of the union of these three lines. Clearly,  any line incident to $P$ will intersect this set of $3d$ points in at most $3$ points. Since $k>2$, we have $d=2^k-1>3$, and so the arc is not complete.

\end{proof}


\section{Acknowledgments }
I thank Felipe Voloch for so many insightful  discussions and comments during the course of this project.
I would like to thank Michael Zieve for pointing out that a result proved in a previous  version of this paper,
namely Lemma \ref{lema1.9}, was known and due to Michael D. Fried and R. E. MacRae. I also thank James Hirschfeld for giving several suggestions to improve the presentation of this manuscript.

The author was partially supported by FAPESP-Brazil grant 2011/19446-3.

\bibliography{mybiblio.bib}
\bibliographystyle{plain}

\end{document}